\documentclass[12pt,leqno]{amsart}
\usepackage{color}
\usepackage{mathtools}
\usepackage[colorlinks=true, pdfstartview=FitV, linkcolor=blue, citecolor=red, urlcolor=blue, backref=page]{hyperref} 

\usepackage[top=1.25in, bottom=1.25in, left=1.25in, right=1.25in]{geometry}

\usepackage{tikz,amsthm,amsmath,amstext,amssymb,amscd,pspicture,multicol,graphpap,graphics,graphicx,enumerate,subfig,sidecap,wrapfig,color,pict2e}

\theoremstyle{plain}
\newtheorem{theorem}{Theorem}

\newtheorem{lemma}[theorem]{Lemma}

\theoremstyle{definition}

\newtheorem{definition}[theorem]{Definition}

\numberwithin{equation}{section}
\numberwithin{theorem}{section}

\newcommand\eqn[1]{\eqref{e:#1}}

\newcommand\Theorem[1]{Theorem \ref{t:#1}}

\newcommand\Lemma[1]{Lemma \ref{l:#1}}

  \newcommand{\R}{\mathbb{R}}

\newcommand\rh{\mathcal{H}}

\newcommand\rd{\mathrm{d}}

  \newcommand{\convexpath}[2]{
[   
    create hullnodes/.code={
        \global\edef\namelist{#1}
        \foreach [count=\counter] \nodename in \namelist {
            \global\edef\numberofnodes{\counter}
            \node at (\nodename) [draw=none,name=hullnode\counter] {};
        }
        \node at (hullnode\numberofnodes) [name=hullnode0,draw=none] {};
        \pgfmathtruncatemacro\lastnumber{\numberofnodes+1}
        \node at (hullnode1) [name=hullnode\lastnumber,draw=none] {};
    },
    create hullnodes
]
($(hullnode1)!#2!-90:(hullnode0)$)
\foreach [
    evaluate=\currentnode as \previousnode using \currentnode-1,
    evaluate=\currentnode as \nextnode using \currentnode+1
    ] \currentnode in {1,...,\numberofnodes} {
-- ($(hullnode\currentnode)!#2!-90:(hullnode\previousnode)$)
  let \p1 = ($(hullnode\currentnode)!#2!-90:(hullnode\previousnode) - (hullnode\currentnode)$),
    \n1 = {atan2(\y1,\x1)},
    \p2 = ($(hullnode\currentnode)!#2!90:(hullnode\nextnode) - (hullnode\currentnode)$),
    \n2 = {atan2(\y2,\x2)},
    \n{delta} = {-Mod(\n1-\n2,360)}
  in 
    {arc [start angle=\n1, delta angle=\n{delta}, radius=#2]}
}
-- cycle
}

\usepackage{latexsym}
\usepackage{tikz}
\usepackage{tkz-euclide}
\usetkzobj{all}
\usetikzlibrary{intersections}
\usetikzlibrary{shapes,arrows,calc}

\renewcommand*{\backref}[1]{}
\renewcommand*{\backrefalt}[4]{%
 \ifcase #1 (Not cited.)%
   \or        (Cited on page~#2.)%
    \else      (Cited on pages~#2.)%
    \fi}

\begin{document}
\allowdisplaybreaks

\title[Bernoulli-type free boundary problem]{On a Bernoulli-type overdetermined free boundary problem}

\author[M. Akman]{Murat Akman}

\address{{\bf Murat Akman}\\
Department of Mathematical Sciences, University of Essex\\
Wivenhoe Park, Colchester, Essex CO4 3SQ, UK
} \email{murat.akman@essex.ac.uk}

\author[A. Banerjee]{Agnid Banerjee}

\address{{\bf Agnid Banerjee}
\\
Tata Institute of Fundamental Research
\\
Center for Applicable Mathematics
\\
Bangalore-560065, India} \email{agnidban@gmail.com}

\author[M. Smit Vega Garcia]{Mariana Smit Vega Garcia}

\address{{\bf Mariana Smit Vega Garcia}
\\
Western Washington University. Department of Mathematics \\
BH230, Bellingham, WA 98225, USA} \email{Mariana.SmitVegaGarcia@wwu.edu}


\date{\today}



\begin{abstract}
In this article we study a Bernoulli-type free boundary problem and generalize a work of Henrot and Shahgholian in \cite{HS1} to $\mathcal{A}$-harmonic PDEs. These are  quasi-linear elliptic PDEs whose structure  is modelled on the $p$-Laplace equation for a  fixed $1<p<\infty$. In particular,  we show that if $K$ is a bounded convex set satisfying the interior ball condition and $c>0$ is a given constant, then there exists a unique convex domain $\Omega$ with $K\subset \Omega$ and a function  $u$ which is  $\mathcal{A}$-harmonic in $\Omega\setminus K$, has continuous boundary values $1$ on $\partial K$ and $0$ on $\partial\Omega$, such that  $|\nabla u|=c$ on $\partial \Omega$. Moreover, $\partial\Omega$ is $C^{1,\gamma}$ for some $\gamma>0$, and it is smooth provided  $\mathcal{A}$ is smooth in $\mathbb{R}^n \setminus \{0\}$. We also show that the super level sets $\{u>t\}$ are convex for $t\in (0,1)$. 
\end{abstract}

\maketitle

\tableofcontents

\section{Introduction and statement of main results}
\label{s1}

Classical Bernoulli free-boundary problems arise in electrostatics, fluid dynamics, optimal insulation, and electro chemistry. In the case of electrostatics, the task is to design an annular condenser consisting of a prescribed boundary component $\partial E$, and an unknown boundary component $\partial \Omega$ (where $\Omega\subset E$), such that the electric field $\nabla u$ is constant in magnitude on the surface $\partial \Omega$ of the second conductor (see \cite{F,FR} for treatment of this problem and applications). This leads to the existence of the following \textit{interior Bernoulli free-boundary} (will be denoted by {\bf (IBFB)}) problem:
\begin{align}
\label{CBFB}
{\bf (IBFB)} \, \left\{
\begin{array}{ll}
-\Delta u=0& \mbox{in}\, \, E\setminus \overline{\Omega},\\
u=1 & \mbox{on}\, \, \partial E, \\
u=0& \mbox{on}\, \, \partial \Omega, \\
|\nabla u|=a & \mbox{on}\, \, \partial \Omega 
\end{array}
\right.
\end{align}
for some given constant $a>0$. The constraint $ |\nabla u|=a$  is called \textit{Bernoulli's law}.  

Here Bernoulli's law $|\nabla u|=a$ should be understood in the following sense:
\[
\liminf\limits_{y\to x, \, y\in E\setminus \overline{\Omega}} |\nabla u(y)|=\limsup\limits_{y\to x, \, y\in E\setminus \overline{\Omega}}|\nabla u(y)|=a \quad  \mbox{for every}\quad x\in\partial\Omega.
\]
The existence and uniqueness of this problem can be stated in the following manner: \textit{is there a domain $\Omega$ with $\Omega\subset E$ and a potential $u:E\setminus \Omega\to \mathbb{R}$ satisfying \eqref{EBFB}? If so, is the couple $(\Omega,u)$ unique?}

The \textit{exterior Bernoulli free-boundary problem} ({\bf (EBFB)} problem for short) is defined in a similar fashion: \textit{is there a couple $(\Omega,u)$ such that $E\subset \Omega$ and \eqref{EBFB} below holds?}
\begin{align}
\label{EBFB}
{\bf (EBFB)} \, \left\{
\begin{array}{ll}
 -\Delta u=0& \mbox{in}\, \, \Omega\setminus\overline{E},\\
u=1 & \mbox{on}\, \, \partial E, \\
u=0& \mbox{on}\, \, \partial \Omega, \\
|\nabla u|=a & \mbox{on}\, \, \partial \Omega.
\end{array}
\right.
\end{align}

In this paper our main goal is to generalize the work of \cite{HS1} on the {\bf (EBFB)} for $\mathcal{A}$-harmonic PDEs (see the definition below in \eqref{AharmonicPDE}), proving 
existence and uniqueness of $(\Omega, u)$ and showing that $\partial \Omega$ is $C^{2,\alpha}$ (see Theorem \ref{t:HS1}). 

Regarding the existing literature, the existence of solutions to the {\bf (EBFB)} problem was obtained by Alt and Caffarelli in \cite{AC} by variational methods, and by Beurling \cite{B} using sub-super solution methods in the plane. The reader is also referred to results of Acker in \cite{Ac2,Ac1} concerning the uniqueness and monotonicity of this problem.  

If we assume $E$ to be convex and require $\Omega$ to be convex as well, the question of existence and uniqueness of a pair satisfying \eqref{EBFB} in the plane was answered affirmatively by Tepper in \cite{T}, by Hamilton in \cite{H} both by using conformal mappings, and by Kawohl in \cite{K}, using different methods. In higher dimensions, convexity and uniqueness of  $\Omega$ were shown by Henrot and Shahgholian in \cite{HS}. Under a convexity assumption on $E$, the {\bf (EBFB)} problem was also studied by Henrot and Shahgholian in \cite{HS1}, where they proved existence of a pair $(\Omega, u)$ satisfying \eqref{EBFB} without assuming $E$ to be bounded or regular. When $E$ is bounded, it was shown in \cite{HS1} that the {\bf (EBFB)} problem has a unique solution, and the same result was obtained independently by Acker and Meyer  in \cite{AM}. 

Neither existence nor uniqueness is always true in the {\bf (IBFB)} case. In the plane, existence of a pair $(\Omega, u)$ satisfying \eqref{CBFB} was obtained by Lavrent{\`e}v in \cite{La}, Beurling in \cite{B}, and Daniljuk in \cite{D}. A higher dimensional result was proved by Alt and Caffarelli in \cite{AC}, and under certain assumptions Henrot and Shahgholian proved  in \cite{HS} that the mean curvature of $\partial\Omega$ is positive for any connected component. In \cite{HS}, it was shown that if the  {\bf (IBFB)} problem admits a solution and $E$ is convex, then $\Omega$ is also convex. 

For further discussion of the problems we consider, we shall first introduce the p-Laplace equation:
\[
\Delta_p u=\nabla \cdot \left(|\nabla u|^{p-2} \nabla u\right) = |\nabla u|^{p-4}[|\nabla u|^{2} \Delta u+ (p-2) \sum\limits_{i,j=1}^{n}u_{x_i} u_{x_j} u_{x_i x_j}].
\]
Here $p$ is fixed with $1<p<\infty$, $|\nabla u|=(u^{2}_{x_1} + \ldots + u^{2}_{x_n})^{1/2}$, and $\nabla\cdot $ is the divergence operator.  

It is well-known that, in general, solutions of the p-Laplace equation do not enjoy second order derivatives in the classical sense, therefore solutions to these equations have to be understood as \textit{weak solutions}. That is, given a bounded, connected open set  $\Omega\subset \mathbb{R}^{n}$, $ u $ 
 is a {\itshape $p$-harmonic function} in $\Omega$ provided $u>0$ in $\Omega$  and $u$ is in the \textit{Sobolev space} $ W^ {1,p} ( U ) $ for each open set $ U $ with  $ \bar U \subset \Omega$ and 
	\begin{align}
	\label{pharmonicPDE}
\int_{U} |\nabla u|^{p-2} \langle \nabla u, \nabla \eta \rangle \rd x=0	\quad  \mbox{whenever}\quad \eta \in W^{1, p}_0 ( U ). 
			\end{align}
In the above paragraph $W^{1,p}(U)$ denotes the space of equivalence classes of functions $h$ with \textit{distributional gradient} $\nabla h$ both of which are $p$ integrable in $U$, and $W^{1,p}_{0}$ denotes the closure of $C^{\infty}_{0}$ in the $W^{1,p}$ norm. 

In \cite{HS2}, the $p$-Laplace operator was treated in the {\bf (IBFB)} case:
\[
\left\{
\begin{array}{ll}
\Delta_{p} u=0 & \mbox{in}\, \, E\setminus \Omega,\\
u=1 & \mbox{on}\, \, \partial E,\\
u=0 & \mbox{on}\, \, \partial \Omega,\\
|\nabla u|=a & \mbox{on}\, \, \partial \Omega. 
\end{array}
\right.
\]
Existence of a solution, and regularity (that is, $\partial\Omega$ is $C^{2,\alpha}$) were shown in that article.

The authors of \cite{HS} proved uniqueness and convexity of the {\bf (EBFB)} problem if the Laplace operator is replaced by a general nonlinear operator $\mathcal{L}$ of the form $\mathcal{L}=\mathcal{L}u=\mathcal{F}(x, u, \nabla u, \nabla^{2} u)$, if the operator $\mathcal{L}$ satisfies certain properties (see section 4 in that article). 

In this article we consider the {\bf (EBFB)} problem when the underlying PDE is the so called {\itshape $\mathcal{A}$-harmonic} PDE. We introduce this nonlinear elliptic equation in what follows. 
\begin{definition}  
\label{defn1.1}	
	Let $p,  \alpha \in (1,\infty) $ be fixed and 
	\[
	\mathcal{A}=(\mathcal{A}_1, \ldots, \mathcal{A}_n) \, : \, \mathbb{R}^{n}\setminus \{0\}  \to \mathbb{R}^{n},
\]
   be such that $  
\mathcal{A}= \mathcal{A}(\eta)$  has  continuous  partial derivatives in 
$ \eta_k$ for $k=1,2,\ldots, n$  on $\mathbb{R}^{n}\setminus\{0\}.$ We say that the function $ \mathcal{A}$ belongs to the class
  $ M_p(\alpha)$ if the following two conditions are satisfied whenever  $\xi\in\mathbb R^n$ and
$\eta\in\mathbb R^n\setminus\{0\}$:
	\begin{align*}
		(i)&\, \, \alpha^{-1}|\eta|^{p-2}|\xi|^2\leq \sum_{i,j=1}^n \frac{
		\partial  \mathcal{A}
_{i}}{\partial\eta_j}(\eta)\xi_i\xi_j\leq\alpha |\eta|^{p-2}|\xi|^2,\\
		(ii) &\, \mathcal{A} \text{ is } p-1 \text{ homogeneous, i.e.,} \,  \mathcal{A} (\eta)=|\eta|^{p-1}  \mathcal{A}
(\eta/|\eta|).
	\end{align*}
	\end{definition}
	
We set  $ \mathcal{A}(0) = 0 $  and note that  Definition \ref{defn1.1}  $(i)$ and $(ii) $ imply   
 \begin{align}  
 \label{eqn1.1}
 \begin{split}
 c^{-1}  (|\eta | + |\eta'|)^{p-2} \,   |\eta -\eta'|^2   \leq   \langle  \mathcal{A}(\eta) -& 
\mathcal{A}(\eta'), \eta - \eta' \rangle \\
&\leq c |\eta-\eta'|^2   (|\eta|+|\eta'|)^{p-2}
\end{split}
\end{align}  
whenever $\eta, \eta'   \in  \mathbb{R}^{n} \setminus \{0\}$.    
We will additionally assume that there exists  $ 1\leq  \Lambda <\infty   $   such that  
\begin{align}
\label{nablaAiscont}
\left| \frac{ \partial \mathcal{A} _i }{\partial \eta_j}   ( \eta )  -  \frac{ \partial \mathcal{A} _{i} }{ \partial\eta_j} ( \eta') \right|\leq  \Lambda  \, | \eta  -   \eta' |    |\eta | ^{p-3} 
\end{align}
whenever $ 0 <|\eta|  \leq 2     |\eta' | $ and  $ 1 \leq  i, j \leq n.$

\begin{definition}\label{def Aharm}
Given an  open set  $ \Omega\subset\mathbb{R}^{n}$ and $\mathcal{A}\in M_p(\alpha)$, one says that $ u$ 
 is $  \mathcal{A}
$-harmonic in $\Omega$, and we write $\nabla \cdot \mathcal{A}(\nabla u)=0$, provided $u>0$ in $\Omega$, $ u \in W^ {1,p} ( U) $ for each open set $ U $ with  $ \bar U \subset \Omega$ and
	\begin{align}
	\label{AharmonicPDE}
		\int \langle    \mathcal{A}
(\nabla u(x)), \nabla \eta (x) \rangle \, \rd x = 0\quad  \mbox{whenever}\quad \eta \in W^{1, p}_0 ( U). 
			\end{align}

			\end{definition}
			For more about PDEs of this type the reader is referred to \cite{HKM}. Notice that when $\mathcal{A}(\eta)=\eta$ then \eqref{AharmonicPDE} is the usual Laplace's equation, and when $\mathcal{A}(\eta)=|\eta|^{p-2} \eta$ then \eqref{AharmonicPDE} becomes the p-Laplace equation.

\begin{definition}
\label{d:unisphere}
We say that a set $K$ satisfies \textit{interior ball condition} if
\begin{align}
\label{e:unisphere}
\mbox{for each}\, \, x_0\in \partial D,\, \, \mbox{there is a ball}\, \,  B(z,\delta)\subset D\, \, \mbox{with}\, \, x_0\in\overline{B(z,\delta)}
\end{align}
for some $\delta>0$.
\end{definition}

Our main goal is to generalize the work of Henrot and Shahgholian in \cite{HS1} on the {\bf (EBFB)} problem to $\mathcal{A}$-harmonic PDEs. In particular, we show that

\begin{theorem}
\label{t:HS1}
Let $c>0$ be a given constant and $K$ be a bounded convex domain satisfying the interior ball condition. Then there exist a unique convex domain $\Omega$ with $K\subset\Omega$ and function $u$ satisfying
\begin{align*}
\begin{cases}
& \mbox{$\nabla\cdot \mathcal{A}(\nabla u)=0$ in $\Omega\setminus \bar{K}$.} \\
& \mbox{$u$ has continuous boundary values $1$ on $\partial K$ and $0$ on $\partial \Omega$.}\\
& \mbox{the superlevel sets $\{u>t\}$ are convex for every $t \in (0,1)$.}\\
&\mbox{$|\nabla u|=c$ on $\partial\Omega$}.
\end{cases}
\end{align*}
Moreover, $\partial\Omega\in C^{1,\gamma}$ for some $\gamma>0$. Furthermore we have that $\partial \Omega$ is smooth provided $\mathcal{A}$ is smooth.  
\end{theorem}


The plan of the paper is as follows. In section \ref{s:preplemmas}, we gather  some known results  concerning the regularity of $\mathcal{A}$-harmonic functions that are relevant for our work. We also show that the levels of $u$ are convex if $K$ is convex by adapting an idea of Lewis \cite{L}. In section \ref{s:proofofHS1}, using a method of Beurling (see also \cite{HS1}), we prove the existence result in Theorem \ref{t:HS1}. Uniqueness   in Theorem \ref{t:HS1} will essentially follow from \cite{HS}. Finally the  regularity result in Theorem \ref{t:HS1} is obtained using ideas inspired by the work of Vogel in \cite{V}.

\section{Notation and Preparatory Lemmas}
\label{s:preplemmas}
Let $x=(x_{1}, \ldots, x_{n})$ denote points in $\mathbb{R}^{n}$ and let $\overline{E},$ $\partial E,$ be the closure and boundary of the set $E\subset\mathbb{R}^{n}.$ Let $\langle \cdot, \cdot\rangle$ be the usual inner product in $\mathbb{R}^{n}$ and $|x|^{2}=\langle x, x\rangle.$  Let $d(E,F)$ denote the distance between the sets $E$ and $F$. Let $B(x, r)$ be the open ball centered at $x$ with radius $r > 0$ in $\mathbb{R}^{n}$ and $dx$ denote the Lebesgue $n-$measure in $\mathbb{R}^{n}$. Given $U \subset \R^n$ an open set and $q$ with $1 \leq q \leq \infty$, let $W^{1,q}(U)$
denote equivalence classes of functions $h:\mathbb{R}^{n}\to\mathbb{R}$ with distributional gradient $\nabla h =\langle h_{x_1}, \ldots, h_{x_n}\rangle$, both of which are $q$-integrable in $U$ with Sobolev norm

\begin{align*}
\|h\|^{q}_{W^{1,q}(U)}=\int\limits_{U}(|h|^{q}+|\nabla h|^{q})\rd x.
\end{align*}
Let $C^{\infty}_{0}(U)$ be the set of infinitely differentiable functions with compact support in $U$ and let $W^{1,q}_{0}(U)$ be the closure of $C^{\infty}_{0}(U)$ in the norm of $W^{1,q}(U)$.

In the sequel, $c$ will denote a positive constant $\geq 1$ (not necessarily the same at each occurrence), which may depend only on $ p, n, \alpha, \Lambda$ unless otherwise stated. In general, $c(a_1, \dots, a_n)$ denotes a positive constant $\geq 1$ which may depend only on $p, n, \alpha, \Lambda, a_1, \dots, a_n$, which is not necessarily the same at each occurrence. 
By $A\approx B$ we mean that $ A/B $ is bounded above and below by positive constants depending only on $p, n, \alpha, \Lambda$.  Finally, in this section we will  always assume that $ 1 < p < \infty$, and $r>0$. 

We next introduce the notion of the \textit{Hausdorff measure}. To this end, let $\hat{r}_{0}>0$ be given, and let $0 <\delta < \hat r_0 $ be fixed. Let $\mbox{diam}(\cdot)$ denote the diameter of a set and let $E \subseteq \mathbb{R}^{n}$ be a given Borel set. For an arbitrary integer $k>0$, we define the $(\delta, k)-$\textit{Hausdorff content} of $E$ in the usual way:
\begin{align*}
 \rh_{ \delta }^{k}(E):= \inf\left\{\sum\limits_{i} r_{i}^{k}:\, \, E\subset \bigcup_{i} B(x_i, r_i) \, \mbox{with}\, \, r_i<\delta\right\}.
 \end{align*}
Here the infimum is taken over all possible covers $\{B(x_i, r_i)\}$ of $E$. Then the \textit{Hausdorff $k$-measure} of $E$ is defined by
\begin{align}
\label{e:hausdorffmeasure}
\rh^{k}(E):= \lim\limits_{\delta \to 0 } \, \,  \rh_{ \delta }^{k}(E).
 \end{align}

\begin{lemma} 
\label{l:fislegthanu}
Given $p$ with $1<p<\infty$, assume that $\mathcal{A}\in M_{p}(\alpha)$ for some $\alpha>1$. Let $u$ be an $\mathcal{A}$-harmonic function in $B(w,4r)$. Then
\begin{align*}
&(a)\, \, r^{p-n}\int\limits_{B(w,r/2)} |\nabla u|^{p} \rd x \leq c\max\limits_{B(w,r)}u^{p}.\\
&(b)\, \,  \max\limits_{B(w,r)} u \leq c\, \min\limits_{B(w,r)} u.
\end{align*}
Moreover, there exists $\gamma\in (0,1)$ depending on $p,n,\alpha$ such that if $x,y\in B(w,r)$ then
\begin{align*}  
\quad \, \, \, \, \, \, (c)\, \, |u(x)-u(y)|\leq c\, (|x-y|/r)^{\gamma} \max\limits_{B(w, 2r)} u\, .
\end{align*}
\end{lemma}
For a proof of Lemma \ref{l:fislegthanu} see \cite{S}. 

\begin{lemma}
\label{l:uisholder}
Given $p$ with $1 < p < \infty$, assume that $\mathcal{A}\in M_{p}(\alpha)$ for some $\alpha>1$. Let $u$ be an $\mathcal{A}$-harmonic function in $B(w,4r)$. Then $u$ has a representative locally in $W^{1,p}(B(w,4r))$ with H{\"o}lder continuous partial derivatives in $B(w,4r)$ (also denoted $u$), and there exist $\beta\in (0,1]$ and $c\geq 1$ depending only on $p,n,\alpha$ such that if $x,y\in B(w,r)$ then
\begin{align}
\label{eqn2.2}
\begin{split}
(i)&\, \,|\nabla u(x)-\nabla u(y)| \leq c (|x-y|/r)^{\beta}\, \max\limits_{B(w,r)}|\nabla u| \leq cr^{-1}(|x-y|/r)^{\beta} u(w). \\
(ii)&\, \, \int_{B(w,r)} \sum\limits_{i,j=1}^{n} |\nabla u|^{p-2} |u_{x_i x_j}|^{2} \rd x \leq c r^{n-p-2}u(w).
\end{split}
\end{align}
Moreover, if 
\[ 
\gamma\, r^{-1} u   \leq |\nabla  u | \leq  \gamma^{-1} r^{-1}   u \quad \mbox{on}\quad B ( w, 2r)  
\] 
for some $ \gamma \in (0,1) $  and  \eqref{nablaAiscont} holds then 
 $u $ has    H\"{o}lder
continuous second  partial derivatives in $B(w,r) $  and
 there exists $\theta \in (0,1), \bar c \geq 1, $ depending only on  the data  and  $ \gamma $  such that if  
  $ x, y \in B (w,  r/2) $,  then   
  \begin{align}  
  \label{eqn2.3}
  \begin{split}
  \Big[ \sum_{i,j = 1}^n  \,   (u_{x_i x_j } (x)   &   -  u_{y_i y_j } (y) ) ^2  \, \Big]^{1/2} \, \leq  \, \bar c  ( | x - y |/ r )^{\theta } \,  \max_{ B(w, r)}\,  \left(\sum_{i,j = 1}^n  \,   |u_{x_i x_j}| \right)     \\ 
  &\leq  \, \bar c^2  r^n  ( | x - y |/ r )^{\theta} \,   \left(\sum_{i,j = 1}^n  \,   \int_{ B ( w, 2r) }   u_{x_i x_j}^2   rd x \right) ^{1/2} \\
  &  \leq  \bar c^3 \,  r^{ - 2} \, ( | x - y |/ r )^{\theta }\, u (w).
\end{split}  
  \end{align}
   \end{lemma}    
   A proof of \eqref{eqn2.2} can be found  in \cite{To}. Estimate \eqref{eqn2.3}  follows from  \eqref{eqn2.2}, the added assumptions  and Schauder type
estimates (see  \cite{GT}).

We will make use of following lemma when we rotate our coordinate system. A proof of it can be found in \cite[Lemma 2.15]{LLN}.

\begin{lemma}
\label{l:rotationissoln}
Let $\Omega\subset\mathbb{R}^{n}$ be a domain and let $p$ with $1<p<\infty$ be given. Let $\mathcal{A}\in M_p(\alpha)$ for some $\alpha>1$ and $u$ be $\mathcal{A}$-harmonic in $\Omega$. If $F:\mathbb{R}^{n}\to \mathbb{R}^{n}$ is the composition of a translation and a dilation, then
\[
\hat{u}(z)=u(F(z)) \, \, \mbox{is $\mathcal{A}$-harmonic in}\, \, F^{-1}(\Omega).
\]
Moreover, if $\tilde{F}:\mathbb{R}^{n}\to \mathbb{R}^{n}$ is the composition of a translation, a dilation, and a rotation then
\[
\tilde{u}(z)=u(F(z)) \, \, \mbox{is $\mathcal{\hat{A}}$-harmonic in}\, \, \tilde{F}^{-1}(\Omega) \quad \mbox{for some} \quad \hat{\mathcal{A}}\in M_p(\alpha).
\]
\end{lemma}
In what follows we make observations that will be useful throughout the paper (see also \cite{ALV12,ALV} for a similar computation). Define 
\begin{equation}\label{mathcalL}
\mathcal{L}(\eta,\xi)=\sum\limits_{i,j=1}^{n}\frac{\partial}{\partial x_i}\left[b_{ij}(\eta)\xi_{x_j}\right], \quad \mbox{where} \quad b_{ij}(\eta)=\frac{\partial \mathcal{A}_i}{\partial \eta_j}(\eta).
\end{equation} 
When $\eta=\nabla u$ we will write $\mathcal{L}(\eta,\xi)=\mathcal{L}_u\xi$, and when $\xi=\nabla w$ we will write $\mathcal{L}_u\xi=\mathcal{L}_u w$.
We next show that if $u$ is $\mathcal{A}$-harmonic in $\Omega$, then $\xi=u$ and $\xi=u_{x_k}$ (for $k=1,\ldots, n$) are both weak solutions to $\mathcal{L}_u\xi=0$.

We first see that if $u$ is $\mathcal{A}$-harmonic then 
\[
\mathcal{L}_u u=\sum\limits_{i,j=1}^{n}\frac{\partial}{\partial x_i}\left[\frac{\partial \mathcal{A}_i}{\partial \eta_j}(\nabla u)u_{x_j}\right]=0.
\]
Indeed, using the $(p-1)$-homogeneity of $\mathcal{A}$ in Definition \ref{defn1.1} we obtain
\begin{align*}
\mathcal{L}_u u&=(p-1) \sum\limits_{i=1}^{n} \frac{\partial}{\partial x_i}\mathcal{A}_i(\nabla u) =(p-1) \nabla \cdot \mathcal{A}(\nabla u)=0.
\end{align*}

To show that $\mathcal{L}_u u_{x_k}=0$ for $k=1,\ldots, n$, using Lemma \ref{l:uisholder} we first get $u\in W^{2,2}(\Omega)$. Then it follows that 
\begin{equation}\label{uxk}
\begin{aligned}
\mathcal{L}_u u_{x_k} &= \sum\limits_{i,j=1}^{n}\frac{\partial}{\partial x_i}\left[\frac{\partial \mathcal{A}_i}{\partial \eta_j}(\nabla u) u_{x_k x_j}\right] = \frac{\partial }{\partial x_k} \nabla \cdot \mathcal{A}(\nabla u)=0.
\end{aligned}
\end{equation}
Note that above argument should be understood in the weak sense. Using these two observations and the structural assumptions on $\mathcal{A}$ from Definition \ref{defn1.1} we also conclude that 
\begin{align*}
\mathcal{L}_u(|\nabla u|^{2}) &=\sum\limits_{i,j=1}^{n}\frac{\partial}{\partial x_i}\left[\frac{\partial \mathcal{A}_i}{\partial \eta_j}(\nabla u) (u^{2}_{x_1}+\ldots +u^{2}_{x_n})_{x_j}\right] =2 \sum\limits_{i,j,k=1}^{n}\frac{\partial}{\partial x_i}\left[\frac{\partial \mathcal{A}_i}{\partial \eta_j}(\nabla u) u_{x_k} u_{x_k x_j}\right]\\
&\stackrel{\eqref{uxk}}{=}2  \sum\limits_{i,j,k=1}^{n} \frac{\partial \mathcal{A}_i}{\partial \eta_j}(\nabla u) u_{x_k x_i} u_{x_k x_j}\geq 2 \alpha^{-1} |\nabla u|^{p-2} \sum\limits_{i,j=1}^{n} (u_{x_i x_j})^{2}.
\end{align*}
Using this observation we conclude
\begin{align}
\label{l:nabla u is subsolution}
\mathcal{L}_u(|\nabla u|^{2})\geq c^{-1} |\nabla u|^{p-2}\sum\limits_{i,j=1}^{n} (u_{x_i x_j})^{2}.
\end{align}

\begin{lemma}  
\label{lemma3.3} Let $\Omega$ be a domain, $K$ be a bounded, closed, convex set with $K\subset\Omega$, $1<p<\infty$ and $\alpha>1$ be given. Let $\mathcal{A}\in M_p(\alpha)$ and $u$ be $\mathcal{A}$-harmonic in $\Omega\setminus K$ with $u=1$ on $\partial K$.  If $K$ satisfies interior ball property then
then there  exists $ c_* \geq 1, $ depending only on $ p, n, \alpha, r_0 $ 
such that if  $ x \in \Omega\setminus K$ 
  \begin{align}
\label{eqn3.12}  
\begin{split}
  &      (a)\quad c_*   \langle \nabla u (x),  z-x \rangle
\geq u (x),\\
& (b) \quad
      c_*^{-1}  \, |x|^{\frac{1-n}{p-1} }   \leq  |  \nabla u (x) |   \leq  c_*   \, |x|^{\frac{1-n}{p-1}} \,  .  
      \end{split} 
\end{align} 
\end{lemma}
\begin{proof}  
A proof of this lemma can be found in \cite{AGHLV} when $1<p<n$ and in \cite{ALSV} when $n\leq p<\infty$. Proof uses Lemmas \ref{l:uisholder}, \ref{l:fislegthanu}, and \ref{l:rotationissoln}. Also a barrier type argument has been used as in \cite[Section 2]{LLN} and \cite[Section 4]{ALV}. We skip the details. 
\end{proof}

Now let  $ u $ be a $\mathcal{A}$-harmonic function in $\Omega\setminus K$ where $K\subset\Omega$, it is  continuous on $ \mathbb{R}^{n} $ with  $ u \equiv 1 $  on  $ K$  and $u\equiv 0$ on $\mathbb{R}^{n}\setminus \Omega$. The following lemma establishes the convexity of the superlevel sets $\{u>t\}$  when both $K$ and $\Omega$ are convex. We note that such a result plays a crucial role in the uniqueness assertion in Theorem \ref{t:HS1}. In the case of $p$-Laplacian, such a result was first established  by Lewis in \cite{L} following Gabriel's \cite{G} ideas. Here we adapt the techniques in \cite{L} and \cite{AGHLV} to our situation. We also refer to the interesting paper \cite{CS} for a different proof in the case of the Laplacian.

\begin{lemma}
\label{lemma3.4}
Let $K\subset \Omega$ be such that $K, \Omega$  are convex and let  $u$ be $\mathcal{A}$-harmonic in $\Omega\setminus K$, continuous on $\R^n$ with $u \equiv 1 $  on  $ K$  and $u\equiv 0$ on $\mathbb{R}^{n}\setminus \Omega$. If $K$ satisfies interior ball condition then for each $ t \in (0, 1)$, the set  $\{ x \in \Omega : u ( x )  >  t \} $ is convex.
\end{lemma}     
 \begin{proof} 
We note from \eqref{eqn3.12} and Lemma \ref{l:uisholder} that
\begin{align}
\label{eqn3.20}
\left\{ 
\begin{array}{l}
 |\nabla u | \neq 0,\\ \mbox{$u$  has H\"{o}lder
 continuous second partial derivatives on compact subsets of  $ \Omega.$ }
 \end{array}
 \right.   
 \end{align}
Our proof of Lemma \ref{lemma3.4} is  by  contradiction, following the proof in \cite[section 4]{L}.  We define for $ \hat x \in \mathbb{R}^{n}, $     
\[
\mathbf{u}( \hat x)=\sup\limits_{\mathclap{\substack{\hat y, \hat z\in\mathbb{R}^{n} \\ \hat x=\lambda \hat y+(1-\lambda) \hat z, \lambda\in[0,1]}}} \ \min\{u(\hat y), u(\hat z)\}.
\]                   

Notice that $u\le\mathbf{u}$, $\mathbf{u}\equiv 1$ in $K$ and $\mathbf{u}\equiv 0$ in $\R^n\setminus \Omega$. It suffices to show that $\mathbf{u}=u$. If that were not true,  then from the convexity of  $ K, $  continuity of  $ u, $  and the fact that as $w\rightarrow w_0\in\partial\Omega$, $u(w)\rightarrow 0$, we would conclude that there must exist $ \epsilon > 0, $ and 
$ x_0  \in  \Omega  $  such that   
\begin{align}
\label{eqn3.21}
0 < \mathbf{u}^{1+\epsilon} (x_0) - u(x_0)  =  \max_{ \mathbb{R}^{n}} (\mathbf{u}^{1+\epsilon}  - u ). 
\end{align}
For ease of  writing we write $ \mathbf{v}  =  \mathbf{u}^{1+\epsilon} $  and  
$ v = u^{1+\epsilon}. $  There exist $\lambda \in (0, 1)  $ and $ y_0, z_0 \in \Omega \setminus \{x_0\}  $ with 
\begin{align}
\label{eqn3.22}   
x_0 =  \lambda y_0 + (1-\lambda)  z_0\quad  \mbox{and}\quad  \mathbf{v}(x_0) = \min\{v(y_0), v(z_0) \}.
\end{align}
We first show that   
\begin{align}
\label{eqn3.23}
v (y_0) = v ( z_0).
\end{align}  
Assume for contradiction, for instance, that $ v (y_0) < v (z_0). $ Since $u\le\mathbf{u}$, $u(x_0)\le u(y_0)=\mathbf{u}(x_0)<u(z_0)$. By continuity, if $z$ is in a small enough neighborhood of $z_0$, then $u(z)>u(y_0)+(u(z_0)-u(y_0))/2$. Since $ |\nabla u|\neq 0$ in $ \Omega, $  we can choose  $ y'$ close enough to $ y_0 $ so that $u(y')>u(y_0)$ and also such that after connecting $y'$ and $x_0$ by a line, we can pick a corresponding $z'$ in the previous neighborhood of $z_0$. In this manner $\mathbf{u}(x_0)\ge \min\{u(y'),u(z')\}>u(y_0)=\mathbf{u}(x_0)$, a contradiction. Thus \eqref{eqn3.23} is true.



Next we prove that 
\begin{align} 
\label{samedirection}
\xi=\frac{\nabla v(y_0)}{|\nabla v(y_0)|}=\frac{\nabla v(z_0)}{|\nabla v(z_0)|}=\frac{\nabla u(x_0)}{|\nabla u(x_0)|}
\end{align} 
Indeed, let us show that  
\[  
\frac{\nabla v(y_0)}{|\nabla v(y_0)|}=\frac{\nabla v(x_0)}{|\nabla v(x_0)|}.
\] 
Let $y$ not be on the line through $y_0$ and $z_0$ and be such that $\nabla v(y_0)\cdot (y-y_0)>0$. Draw the line through $y$ and $z_0$ and denote by $x$ its intersection with the line originating at $x_0$ with direction $y-y_0$. One has that $v(y)>v(y_0)$ for $y$ close to $y_0$. Therefore $\tilde{v}(x_0)\le \tilde{v}(x)$, for $y$ near $y_0$. From \eqref{eqn3.21} we conclude $u(x)\ge u(x_0)$, for $y$ close to $y_0$, hence $\nabla u(x_0)\cdot (y-y_0)\ge 0$ whenever $\nabla v(y_0)\cdot (y-y_0)>0$, showing that $\nabla u(x_0)$ and $\nabla v(y_0)$ point in the same direction.


To simplify our  notation, let
\[
A=|\nabla v(y_0)|,\, \, B=|\nabla v(z_0)|,\, \, C=|\nabla u(x_0)|, \, \, a=|x_0-y_0|, \, \, b=|x_0-z_0|. 
\]
 Let $\eta\in\mathbb{S}^{n-1}$ be such that $\xi\cdot\eta>0$.  From  \eqref{eqn3.20} we can write
\begin{align}
\label{taylorexp}
\begin{split}
v(y_0+\rho\eta)&=v(y_0)+A_1\rho+A_2\rho^{2}+o(\rho^{2}),\\
v(z_0+\rho\eta)&=v(z_0)+B_1\rho+B_2\rho^{2}+o(\rho^{2}),\\
u(x_0+\rho\eta)&=u(x_0)+C_1\rho+C_2\rho^{2}+o(\rho^{2})
\end{split}
\end{align}
as $\rho\to 0$.  Also
\[
A_1/A=B_1/B=C_1/C=\xi\cdot\eta,
\] 
 where  the coefficients and $o(\rho^{2})$ depend on $\eta$. Given $\eta$ with $\xi\cdot\eta>0$ and $\rho_1$ sufficiently small we  see from \eqref{eqn3.20} that the inverse function theorem  can be used to obtain $\rho_2$ with
\[
v\left(y_0+\frac{\rho_1}{A}\eta\right)=v\left(z_0+\frac{\rho_2}{B}\eta\right) .
\]
We  conclude as $\rho_1\to 0$ that
\begin{align}
\label{rho2rho1}
\rho_2=\rho_1+\frac{B}{B_1}\left(\frac{A_2}{A^{2}}-\frac{B_2}{B^{2}}\right) \rho_1^{2}+o(\rho^{2}_1).
\end{align}
Now from geometry we see that $  \lambda  = \frac{b}{a+b} $  so   
\[  x=x_0+\eta\frac{[\rho_1 \frac{b}{A}+\rho_2\frac{a}{B}]}{a+b}
=  \lambda  ( y_0+\frac{\rho_1}{A}\eta ) +  (1-\lambda) (z_0+\frac{\rho_2}{B}\eta).\]
From   this  equality and Taylor's theorem for second derivatives we have 
\begin{align}
\label{taylorforu}
\begin{split}
u(x)-u(x_0)
=& C_1 \, \left[\rho_1\frac{\lambda}{A}+\rho_2\frac{( 1 - \lambda ) }{B} \right]+C_2 \left[\rho_1\frac{\lambda}{A}+\rho_2\frac{(1-\lambda)}{B}\right]^{2}  \\
=& C_1\rho_1 \frac{ ( 1 - \lambda) A + \lambda B}{AB} + C_1 \, \frac{ ( 1 - \lambda ) }{B_1}\left(\frac{A_2}{A^2}-\frac{B_2}{B^2}\right)\rho_1^{2} \\
&+C_2\rho^{2}_1  \left(\frac{(1-\lambda) A+ \lambda B)}{AB} \right)^2               +o(\rho^{2}_1).
\end{split}
\end{align}

Now   
\[
v(y_0+\frac{\rho_1}{A}\eta)-u(x)\leq \mathbf{v}(x)-u(x)\leq \mathbf{v}(x_0)-u(x_0)=v(y_0)-u(x_0).
\]
Hence the mapping 
\[
\rho_1\to v(y_0+\frac{\rho_1}{A}\eta)-u(x)
\]
has a maximum at $\rho_1=0$. Using  the Taylor expansion for $v(y_0+\frac{\rho_1}{A}\eta)$ in \eqref{taylorexp} and $u(x)$ in \eqref{taylorforu} we  have
\begin{align*}
v(y_0+\frac{\rho_1}{A}\eta) -u(x)= &v(y_0)+\frac{A_1}{A}\rho_1+\frac{A_2}{A^{2}}\rho_1^{2}- u(x_0)\\ 
&-C_1\rho_1   \frac{(1 - \lambda) A+\lambda B)}{AB}  - \frac{C_1}{a+b}\frac{a}{B_1}\left(\frac{A_2}{A^2}-\frac{B_2}{B^2}\right)\rho_1^{2}\\
&-C_2\rho^{2}_1       \left(\frac{(1 - \lambda) A+\lambda B)}{AB} \right)^2 +o(\rho^{2}_1)
\end{align*}   Now from the calculus second derivative test,  the  coefficient of $  \rho_1 $  should be zero and the coefficient of  $  \rho_1^2 $  should be non-positive. Hence combining terms we get  
\[  
\frac{A_1}{A}=C_1 \,  \frac{(1 - \lambda) A+\lambda B)}{AB} 
\]   
so taking  $  \eta  = \xi $ we arrive first  at 
\begin{align}
\label{eqn3.28}
\frac{1}{C} \,  = \,  \frac{(1 - \lambda) A+\lambda B)}{AB}    =   \frac{(1 - \lambda)}{B} +  \frac{\lambda}{A} . 
\end{align}
 Second using  \eqref{eqn3.28}  in the  $ \rho_1^2 $  term   we find that 
\begin{align}
\label{eqn3.29}
0\geq \frac{A_2}{A^{2}}  -   C_1  \frac{ (1 - \lambda )}{B_1}\left(\frac{A_2}{A^2}-\frac{B_2}{B^2}\right)-  \frac{C_2}{C^2}\, .
\end{align}
Using  $  C_1/B_1 = C/B $  and doing some arithmetic in \eqref{eqn3.29}  we obtain  
\begin{align}
\label{subnewpde}
0\geq  (1-K)  \frac{A_2}{A^2}+   K \frac{B_2}{B^2}-   \frac{C_2 }{C^{2}}
\end{align}
where
\[
K=\frac{ (1- \lambda)  A}{ ( 1 -  \lambda) A+ \lambda B}<1. 
\]
We now focus on \eqref{subnewpde} by writing $A_1, B_1, C_1$ in terms of derivatives of $u$ and $v$;
\begin{align}
\label{vvu}
0\geq \sum\limits_{i,j=1}^{n} \left[\frac{(1-K)}{A^2} v_{x_i x_j}(y_0)+\frac{K}{B^2}v_{x_i x_j}(z_0)  -\frac{u_{x_i x_j}(x_0)}{C^2} \right]\eta_i \eta_j.
\end{align} 
From symmetry and  continuity  considerations  we  observe that  \eqref{vvu}  holds whenever  $  \eta  \in  \mathbb{S}^{n-1} $  Thus if   
  \[ 
  w (x) = - \frac{(1-K)}{A^2} v (y_0 + x ) - \frac{K}{B^2}v(z_0 + x )\,     + 
\frac{ u_{x_i x_j}(x_0 + x)}{C^2},  
\]  
then the  Hessian matrix of  $ w $ at $ x = 0 $ is positive semi-definite, i.e,   $  ( w_{x_ix_j} (0) ) $  has  non-negative eigenvalues.   Also from  $(i)$ of  Definition \ref{defn1.1} we see that  if 
\[
 a_{ij} = \frac{1}{2}\left[ \frac{\partial \mathcal{A}_i}{\partial \eta_j}  (\xi ) )    
+   \frac{\partial \mathcal{A}_j}{\partial \eta_i}  (\xi ) )   \right],\, \, 1 \leq i, j \leq  n,
\]
then  $  (a_{ij}) $  is  positive definite.  From  these two observations we conclude that 
\begin{align}
\label{eqn3.32}   
\mbox{trace} \left( (  (a_{ij} )  \cdot ( w_{x_ix_j} (0) ) \right)   \geq  0.   
\end{align}   
To obtain a contradiction we observe from \eqref{AharmonicPDE},  the divergence theorem,  \eqref{subnewpde}, and   
  $  p - 2 $  homogeneity of  partial derivatives of  $ \mathcal{A}_i , $      that 
  \begin{align}
  \label{eqn3.33}    
  \sum\limits_{i,j=1}^{n}  a_{ij} \,  u_{x_jx_i} =   | \nabla u |^{2-p}  \sum\limits_{i,j=1}^{n} \frac{\partial \mathcal{A}_i}{\partial \eta_j}(\nabla u ) u_{x_jx_i} \, = \, 0 \mbox{ at } x_0, y_0, z_0. 
\end{align}  
Moreover   from the definition of $v$ we have 
\begin{align}
\label{secondderiv}
\begin{split}
&v_{x_{i}}=(u^{1+\epsilon})_{x_{i}}= (1+\epsilon)u^{\epsilon} u_{x_{i}} \\
&v_{x_{i}x_{j}}= (1+\epsilon)\epsilon u^{\epsilon-1} u_{x_{i}}u_{x_j}+(1+\epsilon)u^{\epsilon} u_{x_{i} x_{j}}.
\end{split}
\end{align}
Using \eqref{eqn3.33}, \eqref{secondderiv}, we find that 
\begin{align}
\label{fvupos}
\begin{split}
   |\nabla u |^{p-2} \, \sum\limits_{i,j=1}^{n}  & 
  a_{ij}  \,  v_{x_j x_i}  = \sum\limits_{i,j=1}^{n}  \frac{\partial \mathcal{A}_i}{\partial \eta_j}(\nabla u)  [(1+\epsilon)u^{\epsilon-1} u_{x_{j}}u_{x_i}+(1+\epsilon)u^{\epsilon} u_{x_{j} x_{i}}]\\
&=(1+\epsilon)\epsilon u^{\epsilon-1}\sum\limits_{i,j=1}^{n}  \frac{\partial \mathcal{A}_i}{\partial \eta_j}(\nabla u)  u_{x_{j}}u_{x_i} + (1+\epsilon)u^{\epsilon}   \,  \sum\limits_{i,j=1}^{n} \frac{\partial \mathcal{A}_i}{\partial \eta_j}(\nabla u )  u_{x_{j}x_i}\\
& \geq \alpha^{-1} (1+\epsilon)\epsilon u^{\epsilon-1}  |\nabla u|^{p-2} |\nabla u|^{2} +0>0
\end{split}
\end{align}
at points $y_0$ and $z_0$ ($\nabla u$ is also evaluated at these points). Using  \eqref{eqn3.33}, \eqref{fvupos}, we conclude that   
\begin{align} 
\label{eqn3.36}
\mbox{trace} \left(  (a_{ij} ) \cdot ( w_{x_ix_j} (0) ) \right)          =    \sum_{i,j=1}^n   a_{ij}   \,  w_{x_ix_j} (0)        < 0.
\end{align}
  \eqref{eqn3.36} and  \eqref{eqn3.32} contradict each other.    The proof of  Lemma \ref{lemma3.4} is now complete. 
\end{proof}  

\section{Proof of \Theorem{HS1}}
\label{s:proofofHS1}
In this section we give a proof of Theorem \ref{t:HS1} by a method of \textit{Beurling}, inspired by Henrot and Shahgholian in \cite{HS1}. To this end, Let $K$ be a convex domain and let $P_{x_{0},a}$ denote the hyperplane in $\mathbb{R}^{n}$ passing through $x_0$ with the normal $a\neq 0$ pointing away from $K$. 

A \textit{supporting hyperplane to $K$ at boundary point $x_{0}$} is a plane satisfying
\[
P_{x,a}:=\{x: \, \, a^{T}x=a^{T}x_{0}\}
\]
where $a\neq 0$ and $a^{T}x\leq a^{T}x_0$ for all $x\in K$. By the supporting hyperplane theorem it is known that there exists a supporting hyperplane at every boundary point of a convex set $K$. Let $\Omega$ be another convex set containing $K$. 

For each $x\in \partial K$ there exists a point $y_{x}\in \partial \Omega\cap \{z:\,\, a\cdot(z-x)>0\}$
satisfying $a\cdot (y_{x}-x)=\max a\cdot (z-x)$, where maximum is taken over the set $\partial \Omega\cap \{z:\,\, a\cdot(z-x)>0\}$. 




We will work on convex ring domains. That is, let $D_1$ and $D_2$ be two convex domains satisfying $D_1\subset \overline{D_1}\subset D_2$. We first need an auxiliary lemma.  
\begin{lemma}
\label{l:lemma2.2HSI}
Let $D_1, D_2$ be two convex domains with $D_1\subset \overline{D_1}\subset D_2$. Let $u$ be $\mathcal{A}-$capacitary potential of $D_2\setminus D_1$, that is,
\[
\left\{
\begin{array}{ll}
\nabla\cdot \mathcal{A}(\nabla u)=0 & \mbox{in}\ D_2\setminus\overline{D_1},\\
u=c_1 &\mbox{on}\ \partial D_1,\\
u=c_2 & \mbox{on}\ \partial D_2,
\end{array}
\right.
\]
where $c_1>c_2\geq 0$ are given constants. Then
\[
\limsup\limits_{\substack{z\to x\\z\in D_2\setminus\overline{D_1}}} |\nabla u(z)|\geq \limsup\limits_{\substack{z\to y_x\\ z\in D_2\setminus\overline{D_1}}} |\nabla u(z)|
\]
for all $x\in \partial D_1$. 
\end{lemma}

\begin{proof}
Without loss of generality assume that $c_1=1$ and $c_2=0$ as we may using the the translation and dilation invariance of \eqref{AharmonicPDE}. Now let $x \in \partial D_1$ and also first assume that $\partial D_1$ is not $C^{1}$ at $x$.  Note that locally near $x$, $u$ can be approximated by functions $u^{\epsilon}$, which are solutions to a uniformly elliptic PDE in non-divergence form  with ellipticity bounds independent of $\epsilon$ (see \cite[section 2.3]{A}). This later fact follows from the structural assumptions on $\mathcal{A}$ as in $(i)$ in  Definition \ref{defn1.1}. Then it follows from  \cite{Mi} that there exists a barrier $v$ to such linear equations  with $v(x)=u^{\epsilon}(x)= u(x)$, $v \leq u^{\epsilon}$ near $x$ and moreover $|\nabla v(x)|= \infty$. Thus it follows that $ |\nabla u(x)|=\infty$. Likewise, if $\partial D_2$ is not $C^{1}$ at $y_x$, then similarly it follows  from \cite{Mi} that there exists   an upper barrier $v$ such that $v(y_x)= u(y_x)$ and $v \geq u$ locally near $y_x$ and such that $|\nabla v(y_x)|=0$. Then it follows that $|\nabla u(y_x)|=0$, which gives the desired result. Thus in view of the above discussion, we may now restrict our attention to the case when $x, y_x$ are in the regular part of $\partial D_1$ and $\partial D_2$ respectively.  Let $x\in \partial D_1$ be fixed and let $y_x$ be the associated point on $\partial D_2$ as described above. Let $P=P_{x,a}$ be a supporting plane at $x$ to $D_1$. 

\begin{center}
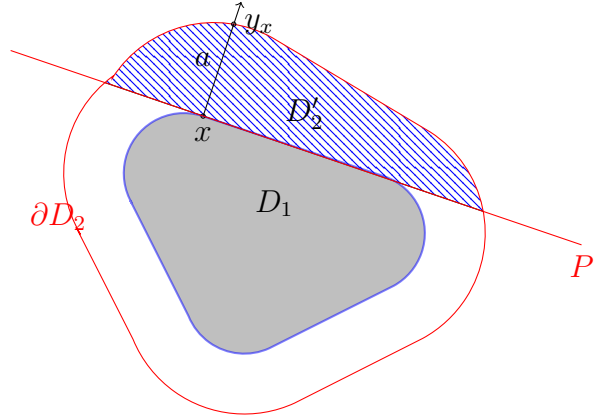
\begin{figure}[!ht]
\begin{tikzpicture}[scale=.8]
\usetikzlibrary{patterns}

  \node at (0,0) (a) {};
  \node at (.5,.5) (b) {};
  \node at (3,-1) (c) {};
  \node at (1,-2) (d) {};
\begin{scope}
\clip \convexpath{a,b,c,d}{2cm};
\draw[pattern=north west lines, pattern color=blue] (-2.8822, 2.0395)--(6.6045, -1.193)--(6.6045, 5)--(-2.8822,5)--(-2.8822, 2.0395);
\end{scope} 
\draw[thick,blue, opacity=.5, fill=gray] \convexpath{a,c,d}{1cm};
\draw[red, name path=line1] \convexpath{a,b,c,d}{2cm} node[left] {$\partial D_2$};

\node at (1.5,-.5) {$D_1$};
\draw (0.31622, 0.9486) circle (1pt) node[below] {$x$};
\draw[name path=line2, ->] (0.31622, 0.9486)--(0.9486, 2.8460) node[midway, left] {$a$};
\draw [name intersections={of=line1 and line2, by=x}] (x) circle (1pt)node[right] {$y_{x}$};
\node at (2,1) {$D'_2$};
\begin{scope}
\draw[red] (-2.8822, 2.0395)--(6.6045, -1.193) node[below] {$P$};

\end{scope}
\end{tikzpicture}
\caption{The supporting place $P=P_{x,a}$ and the domains $D_1$ and $D_2$.}
\end{figure}
\end{center}
\vspace{-1cm}

Note that $D_1\subset \{P<0\}$ and let $D'_2:=D_2\cap \{P>0\}$. By Lemma \ref{l:rotationissoln}, we may assume that $P=\{x_n=0\}$.  Indeed, otherwise after a rotation we have $P=\{x_n=0\}$; we first prove the present lemma for $\tilde{u}$ which is $\tilde{\mathcal{A}}$-harmonic for some $\tilde{\mathcal{A}}\in M_p(\alpha)$ and follow by transferring everything back to $u$. Hence assume $P=\{x_n=0\}$ and  define $v=u+\alpha x_n$, where 
\[
\alpha= \limsup\limits_{\substack{z\to y_x\\ z\in D_2\setminus\overline{D_1}}} |\nabla u(z)|-\epsilon,
\]
with $\epsilon>0$ small. Since $\mathcal{L}u=\mathcal{L}v=0$ in $D'_2$, $v$ attains its maximum on $\partial D'_2$. By the construction of $D'_2$, the maximum of $v$ is either at $x$ or $y_x$. If the maximum were at $y_x$, then
\[
0\leq \limsup\limits_{\substack{z\to y_x\\ z\in D_2\setminus\overline{D_1}}} \frac{\partial v}{\partial x_n}(z)=-\limsup_{\substack{z\to y_x\\ z\in D_2\setminus\overline{D_1}}}|\nabla u(z)|+\alpha=-\epsilon<0.
\]
It follows that $v$ attains its maximum at $x$, hence
\begin{align*}
0&\leq -\limsup_{\substack{z\to x\\ z\in D_2\setminus\overline{D_1}}} \frac{\partial v}{\partial x_n}(z)=\limsup_{\substack{z\to x\\ z\in D_2\setminus\overline{D_1}}}|\nabla u(z)|-\alpha\\
&=\limsup_{\substack{z\to x\\ z\in D_2\setminus\overline{D_1}}}|\nabla u(z)|-\limsup\limits_{\substack{z\to y_x\\ z\in D_2\setminus\overline{D_1}}} |\nabla u(z)|+\epsilon.
\end{align*}
We then conclude the validity of the lemma.
\end{proof}

We next show that if $D_1$ satisfies the so called \textit{the interior ball property} as in \eqref{e:unisphere} then the $\mathcal{A}$-capacitary function $u$ as above has bounded gradient. 

\begin{lemma}
\label{l:lemma2.3HSI}
Let $D_1, D_2$ be as in Lemma \ref{l:lemma2.2HSI} and let $d_0=\min d(\partial D_2, D_1)$. Assume also that $D_1$ satisfies the interior ball property as in \eqref{e:unisphere} with constant $r_0$. Then there is a constant $M=M(d_0, r_0, n)$ such that
\[
|\nabla u|\leq M\, \, \mbox{in} \, \,  D_2\setminus \overline{D_1}. 
\]
\end{lemma}

\begin{proof}
In view of \eqref{l:nabla u is subsolution}, it is enough to show that $|\nabla u|\leq M$ on $\partial D_1\cup \partial D_2$. 

We first take care of points on $\partial D_2$. Without loss of generality take $c_1=1, c_2=0$ as \eqref{AharmonicPDE} is invariant under translation and dilation. Let $x\in \partial D_2$ be fixed. By rotation, assume, by Lemma \ref{l:rotationissoln}, that $x_n=0$ is a supporting hyperplane to $\partial D_2$ at $x$ with $D_2\subset \{x_n>0\}$ and prove the present lemma for $\tilde{u}$ which is $\tilde{\mathcal{A}}$-harmonic for some $\tilde{\mathcal{A}}\in M_p(\alpha)$ and transfer the result back to $u$. Therefore, without loss of generality $x_n=0$ is a supporting hyperplane.  There exists a supporting hyperplane $x_n=d$ to $\partial D_1$ for which $D_1\subset \{x_n>d\}$. Let $\tilde{D}_2=D_2\cap \{ 0<x_n<d\}$ and let $\tilde{u}=x_n/d$. Basic comparison principle applied to positive weak solutions of $\mathcal{A}$-harmonic PDEs gives $u\leq \tilde{u}$ in $\tilde{D}_2$. This observation and the fact that $u(x)=\tilde{u}(x)$ implies
\[
|\nabla u(x)|\leq |\nabla\tilde{u}(x)|\leq \frac{1}{d}\leq\frac{1}{d_0}.
\]
This gives the desired results for points on $\partial D_2$. 

In order to show the same estimates for points on $\partial D_1$, we proceed as follows. We first construct a barrier as we did in the proof of Lemma \ref{lemma3.3} and then we prove that Lemma \ref{l:lemma2.2HSI} holds for $u_{\epsilon}$. Finally, using Lemmas \ref{l:fislegthanu} and \ref{l:uisholder} we conclude that Lemma \ref{l:lemma2.2HSI} holds for $u$ as well. 
\end{proof}

\subsection{A technique of Beurling}
In this subsection we give a brief introduction to a technique used by Beurling in \cite{Be} and in \cite{HS1} as well. To this end, recall that $K$ is a convex domain and let
\[
\mathcal{C}:=\{\Omega\, \, \mbox{convex bounded open subset of}\, \, \mathbb{R}^{n}\, \, \mbox{with}\, \, K\subset\Omega\}.
\]
Let $u_\Omega$ denote the $\mathcal{A}-$capacitary potential for $\Omega\setminus K$ whenever $\Omega\in \mathcal{C}$. Following  \cite{HS2}, we also define 
\[
\begin{array}{r}
\mathcal{G}:=\{\Omega\in\mathcal{C}:\,\, \liminf\limits_{y\to x, y\in\Omega} |\nabla u_{\Omega}(y)|\geq c\, \, \mbox{for all}\, \, x\in\partial\Omega\},\\
\mathcal{G}_0:=\{\Omega\in\mathcal{C}:\,\, \liminf\limits_{y\to x, y\in\Omega} |\nabla u_{\Omega}(y)|> c\, \, \mbox{for all}\, \, x\in\partial\Omega\},\\
\mathcal{B}:=\{\Omega\in\mathcal{C}:\,\, \limsup\limits_{y\to x, y\in\Omega} |\nabla u_{\Omega}(y)|\leq c\, \, \mbox{for all}\, \, x\in\partial\Omega\}.
\end{array}
\]
In the language of Beurling, $\mathcal{G}$ is the collection of ``subsolutions'' and $\mathcal{B}$ is the collection of ``supersolutions''. Our aim is to show that $\mathcal{G}\cap\mathcal{B}\neq\emptyset$. To this end, we will make some observations.


\begin{lemma}
\label{l:lemma-HS3.1}
$\mathcal{B}$ is closed under intersection. That is, if $\Omega_1,\Omega_2\in\mathcal{B}$ then $\Omega_1\cap\Omega_2\in\mathcal{B}$.
\end{lemma}
\begin{proof}
We will use the comparison principle for non-negative $\mathcal{A}$-harmonic functions. Let $u_{\Omega_i}$ for $i=1,2$ be $\mathcal{A}-$capacitary functions for $\Omega_i\in\mathcal{B}$. By the comparison principle, we have $u_{\Omega_1\cap \Omega_2}\leq \min\{u_{\Omega_1}, u_{\Omega_2}\}$ in $(\Omega_1\cap\Omega_2)\setminus K$. Furthermore, $\partial(\Omega_1\cap\Omega_2)\subset \partial\Omega_1\cap\partial\Omega_2$, hence given $x\in \partial(\Omega_1\cap\Omega_2)$ we can assume without loss of generality $x\in\partial\Omega_1$. Then $u_{\Omega_1}(x)=0=u_{\Omega_1\cap \Omega_2}(x)$ and thereupon one concludes that
\[
\limsup\limits_{y\to x, y\in\Omega_1\cap\Omega_2} |\nabla u_{\Omega_1\cap \Omega_2}(y)|\leq \limsup\limits_{y\to x, y\in\Omega_1} |\nabla u_{\Omega_1}|\leq c.
\] 
Therefore $\Omega_1\cap\Omega_2\in \mathcal{B}$. This finishes the proof of Lemma \ref{l:lemma-HS3.1}.
\end{proof}
Our next goal is to show the ``stability'' of $\mathcal{B}$. 
\begin{lemma}
\label{l:lemma-HS3.2}
Assume $K$ satisfies the interior ball property. Let $\Omega_1\supset \Omega_2\supset\ldots$ be a decreasing sequence of domains in $\mathcal{B}$. Let 
\[
\Omega=\mathring{\overline{\bigcap\limits_{n}\Omega_n}}.
\]
Assume $\Omega\in\mathcal{C}$. Then $\Omega\in \mathcal{B}$.
\end{lemma}
\begin{proof}
Let $\Omega_1\supset \Omega_2\supset\ldots$ be a sequence of domains in $\mathcal{B}$ and let $\{u_k\}$ be a sequence of capacitary $\mathcal{A}-$harmonic functions for $\{\Omega_k\}$ respectively. Then $0\leq u_k\leq 1$. Moreover, by Lemmas \ref{l:fislegthanu}, \ref{l:uisholder}, and \ref{l:lemma2.3HSI} it follows that $\{u_k,\nabla u_k\}$ converges uniformly on compact subsets of $\Omega\setminus\overline{K}$ to  $\{u,\nabla u\}$ where $u$ is a  $\mathcal{A}-$harmonic function in $\Omega\setminus K$.  The proof that $u$ is indeed the capacitary $\mathcal{A}-$harmonic function for $\Omega\setminus K$ essentially follows from the convergence of $\Omega_n$ to $\Omega$ in the Hausdorff distance sense and Lemma \ref{lemma3.3}. 

We next show that $\Omega\in \mathcal{B}$. To this end, let $M=\max_{k}(\sup|\nabla u_k|)<\infty$, by Lemma \ref{l:lemma2.3HSI}. Let $0<\delta_k$  be such that $\delta_k\to 0$ as $k\to\infty$ and
\[
\frac{|\nabla u_k|^{2}-c^2}{M^2}-\frac{1}{k} \leq 0 \quad \mbox{on} \quad \{u_k=\delta_k\}.
\]
Consider 
\[
\frac{u_k-\delta_k}{1-2\delta_k},
\]
which is non-negative in $\{u_k>\delta_k\}\setminus \{u_k<1-\delta_k\}$ and has zero boundary values on  $\{u_k=\delta_k\}$. Recall definition \eqref{mathcalL}. By (\ref{l:nabla u is subsolution}) applied to
$\frac{|\nabla u_k|^{2}-c^2}{M^2}-\frac{1}{k}$,
\[
\mathcal{L}_{u_k}\left(\frac{|\nabla u_k|^{2}-c^2}{M^2}-\frac{1}{k}\right)\geq 0 = \mathcal{L}_{u_{k}}\left(\frac{u_k-\delta_k}{1-2\delta_k}\right)
\]
in $\{u_k>\delta_k\}\setminus \{u_k<1-\delta_k\}$. On the other hand, on $ \{u_k=\delta_k\}$ we have
\[
\frac{|\nabla u_k|^{2}-c^2}{M^2}-\frac{1}{k} \leq 0=\frac{u_k-\delta_k}{1-2\delta_k}.
\] Furthermore, on $\{u_k=1-\delta_k\}\cap\{u_k\ge\delta_k\}$ we have
\[
\frac{|\nabla u_k|^{2}-c^2}{M^2}-\frac{1}{k} \leq 1=\frac{u_k-\delta_k}{1-2\delta_k}.
\]
It follows that
\begin{equation}\label{bound}
\frac{|\nabla u_k|^{2}-c^2}{M^2}-\frac{1}{k} \leq \frac{u_k-\delta_k}{1-2\delta_k} \quad\mbox{in} \quad \{u_k>\delta_k\}\setminus \{u_k<1-\delta_k\}.
\end{equation}
Given $\epsilon>0$ one can find a neighborhood $U_\epsilon$ of $\partial\Omega$ such that 
\[
u_k\leq \epsilon \quad \mbox{in}\quad U_\epsilon.
\]
Letting $k\to\infty$ and using \eqref{bound} we obtain
\[
\frac{|\nabla u_k|^{2}-c^2}{M^2}-\frac{1}{k} \to \frac{|\nabla u|^{2}-c^2}{M^2}\leq u\leq \epsilon \quad \mbox{as} \quad k\to\infty
\]
uniformly on compact subsets of $U_\epsilon\cap\Omega$. By letting $\epsilon\to 0$ we conclude the proof of the Lemma. 
\end{proof}
As a consequence of Lemma \ref{l:lemma-HS3.2}, we claim that if $\mathcal{G}_0$ is not empty and for $\Omega_0\in \mathcal{G}_0$ the set $\{\tilde{\Omega}\in\mathcal{B};\, \, \overline{\Omega_0}\subset \tilde{\Omega}\}$ is not empty, then there exists a domain $\Omega\in\{\tilde{\Omega}\in\mathcal{B};\, \, \overline{\Omega_0}\subset \tilde{\Omega}\}$ with the property that 
\begin{align}
\label{minimal}
\mbox{if $\hat{\Omega}\in \{\tilde{\Omega}\in\mathcal{B};\, \, \overline{\Omega_0}\subset \tilde{\Omega}\}$ and $\hat{\Omega}\subset \Omega$, then $\hat{\Omega}=\Omega$.}
\end{align} 

Such a domain $\Omega$ will be called \emph{minimal element} in $\{\tilde{\Omega}\in\mathcal{B};\, \, \overline{\Omega_0}\subset \tilde{\Omega}\}$. For simplicity, define 
\[
\mathcal{C}_0:=\{\tilde{\Omega}\in\mathcal{B};\, \, \overline{\Omega_0}\subset \tilde{\Omega}\}.
\]
To prove our claim, let 
\[
I=\bigcap_{i} \tilde{\Omega}_i \quad \mbox{where} \quad \tilde{\Omega}_i\in\mathcal{C}_0.
\]
Write $I=\cap_{i=1}^{\infty}\tilde{\Omega}_i$, with $\tilde{\Omega}_i\in\mathcal{C}_0$. Let $\Omega_1=\tilde{\Omega}_1$ and $\Omega_{k+1}=\tilde{\Omega}_{k+1}\cap\Omega_k$ for $k=2,3,\ldots$. Then each $\Omega_k$ is convex and $\Omega_k\in\mathcal{B}$ by Lemma \ref{l:lemma-HS3.1}. Applying Lemma \ref{l:lemma-HS3.2} to $\{\Omega_n\}$ 
we conclude that
\[
\Omega=\mathring{\overline{\bigcap\limits_{n}\Omega_n}}\in \mathcal{B}.
\]
This finishes the proof of our claim. 

We proceed by studying the behavior of capacitary $\mathcal{A}$-harmonic functions on extremal points of $\Omega$.  To set the stage, let $\Omega$ be the minimal element in $\mathcal{C}_0$. A point $x\in\partial\Omega$ is called \textit{extremal point} if there exists a supporting hyperplane to $\Omega$ touching $\partial\Omega$ at $x$ only. Let $E_\Omega$ denote the set of extremal points of $\Omega$. 

\begin{lemma}
\label{l:lemma-HS3.4}
Let $\Omega$ be a minimal element in the class $\mathcal{C}_0$ and let $x\in\overline{E_{\Omega}}$. Then
\[
\limsup\limits_{\substack{y\to x\\ y\in\Omega}} |\nabla u_{\Omega}(y)|=c.
\]
\end{lemma}
\begin{proof}
The proof will be a contradiction argument. To this end, suppose there exists $y_0\in \overline{E_\Omega}$ with
\[
\limsup\limits_{\substack{y\to y_0\\ y\in\Omega}} |\nabla u_{\Omega}(y_0)|=c(1-4\tilde{\alpha}).
\]
for some $\tilde{\alpha}>0$. By the H{\"o}lder continuity of $\nabla u_\Omega$ there exists a neighborhood $\mathcal{N}$ of $\partial\Omega$ with $y_0\in \mathcal{N}$ satisfying that
\begin{align}
\label{e:3.5}
|\nabla u_\Omega(x)|\leq c(1-\tilde{\alpha})\, \, \mbox{for every}\, \, x\in \mathcal{N}\cap\Omega.
\end{align}
Assume that $y_0\in E_\Omega$. Otherwise, we may choose a sequence in $E_{\Omega}$ converging to $y_0$ with the above property. 







Let $d>0$ and let $P_d$ be a plane such that $d(y_0, P_d)=d$ with $P_d \cap \Omega \subset\mathcal{N}$. Notice that without loss of generality, by Lemma \ref{l:rotationissoln} we may assume that $y_0=0$, $P_d =\{x_n=d\}$. Otherwise, we rotate our coordinate system, work with $\hat{u}$, which is $\hat{\mathcal{A}}$-harmonic for some $\hat{\mathcal{A}}\in M_p(\alpha)$, and at the end transfer everything back to $u_{\Omega}$. 

Hence assume $y_0=0$, $P_d =\{x_n=d\}$, let $\epsilon>0$ and define $\Omega_\epsilon=\Omega\setminus \{x_n\le\epsilon\}$. Assume $\epsilon$ is small enough so that $\Omega_0\subset \Omega_{\epsilon}$. Let $u_\epsilon$ be the $\mathcal{A}$-capacitary function for $\Omega_\epsilon\setminus K$. As $u_\epsilon\leq u_\Omega$ on $\partial \Omega_\epsilon$, by the comparison principle for non-negative $\mathcal{A}$-harmonic functions we have 
\begin{align}
\label{e:3.6}
0\leq u_\epsilon\leq u_\Omega\, \, \mbox{in}\, \, \Omega_\epsilon.
\end{align}
It follows that we have
\begin{equation}\label{nicepart}
\limsup |\nabla u_{\epsilon}|\leq \limsup|\nabla u_\Omega|\leq c \, \, \mbox{on}\, \, \partial\Omega\cap\partial\Omega_\epsilon.
\end{equation}
At the points where $\partial\Omega\cap\partial\Omega_\epsilon$ is not $C^{1}$ we claim that $|\nabla u_{\epsilon}|=0$. This can be done 
as in \cite[Section 7]{AGHLV} by considering $\mathcal{A}(\eta,\delta)$ to obtain a uniformly elliptic equation in divergence form and $v^{\delta}_\epsilon$ which is $\mathcal{A}(\eta,\delta)$-harmonic in $\Omega_\epsilon$. Once again repeating following \cite[Section 7]{AGHLV}, one concludes that $|\nabla v_{\epsilon}^{\delta}|=0$ and thereupon letting $\delta\to 0$ we obtain our claim. 

Using \eqn{3.5} and \eqn{3.6} we have 
\begin{align}
\label{e:limsupnablauepsilon}
\max\limits_{P_d\cap\Omega} u_\epsilon\leq \max\limits_{P_d\cap\Omega} u_\Omega\leq d \sup\limits_{\{0\leq x_n\leq d\}\cap\Omega} |\nabla u_\Omega|\leq d\, c(1-\tilde{\alpha})
\end{align}
Let now
\[
v:=u_\epsilon+\frac{d\,c(1-\tilde{\alpha})}{d-\epsilon}(d-x_n)
\]
and note that
\[
\mathcal{L}_{u_\epsilon}v=\mathcal{L}_{u_\epsilon}u_\epsilon=0\, \, \mbox{in}\, \, \Omega_\epsilon\cap\{x_n<d\}.
\]
Thereupon we conclude that $v$ takes its maximum on the boundary of $\Omega_{\epsilon}\cap\{x_n<d\}$. Moreover, using \eqref{e:limsupnablauepsilon} we obtain
\[
v\leq c(1-\tilde{\alpha})d \, \, \mbox{on}\, \, \partial(\Omega_\epsilon\cap\{ x_n<d\})\, \, \mbox{and}\, \,  v=d\, c(1-\tilde{\alpha})\, \, \mbox{on}\, \, P_\epsilon,
\]
as 
\[
\partial(\Omega_\epsilon\cap\{ x_n<d\})\subset P_d\cup P_\epsilon\cup (\partial\Omega\cap\{\epsilon<x_n<d\})
\]
Hence we have
\[
0\geq \frac{\partial v}{\partial x_n}=|\nabla u_\epsilon|-\frac{c(1-\tilde{\alpha})d}{d-\epsilon}\, \, \mbox{on}\, \,  P_\epsilon. 
\]
By choosing $\epsilon\leq \tilde{\alpha} d$, we obtain $|\nabla u_\epsilon|\leq c$ on $P_\epsilon$. In view of this result and \eqref{nicepart} we have $\Omega_\epsilon\in \mathcal{B}$ and by construction $\Omega_\epsilon\subset\Omega$. By \eqref{minimal} we conclude that $\Omega=\Omega_\epsilon$ which is a contradiction, hence the proof of Lemma \ref{l:lemma-HS3.4} is complete.
\end{proof}

We next observe that if $\Omega$ is a minimal element in the class $\mathcal{C}_0$ and let $u_\Omega$ be the $\mathcal{A}$-capacitary function for $\Omega\setminus K$ then 
\begin{align}
\label{e:lemma-3.5}
|\nabla u_\Omega(x)|\geq c\, \, \mbox{for all}\, \, x\in\Omega\setminus K.
\end{align}
To prove \eqref{e:lemma-3.5} we use the fact that for every $0<t<1$, $\{x\in\Omega : u_{\Omega}>t\}$ is a convex  set due to Lemma \ref{lemma3.4}. The conclusion follows by applying Lemma \ref{l:lemma2.2HSI} to $\{x\in \Omega: u_{\Omega}>t\}$ and $\Omega$, and using Lemma \ref{l:lemma-HS3.4}. 

\subsection{Final Proof of Theorem \ref{t:HS1}}
We split the proof into two steps, existence of $\Omega$ and uniqueness of $\Omega$. 

\subsubsection{Existence of $\Omega$:} In order to prove Theorem \ref{t:HS1} we show that there exist domains $\Omega_0$ and $\Omega_1$ such that $\overline{\Omega_0}\in \mathcal{G}_0$ and $\Omega_1\in \mathcal{B}$ with $\Omega_0\subset\Omega_1$. Then from \eqref{minimal} there exists a minimal element $\Omega\in\mathcal{C}_0$ and by using \eqref{e:lemma-3.5} we have $\Omega\in \mathcal{G}$. In view of the definitions of $\mathcal{G}$ and $\mathcal{B}$, this would lead to the existence portion of Theorem \ref{t:HS1} is true. Hence to finish the proof of existence, it remains to show the existence of $\Omega_0\in \mathcal{G}_0$ and $\Omega_1\in\mathcal{B}$ with $\overline{\Omega_0}\subset \Omega_1$. \\

\noindent {\bf Existence of $\Omega_1\in \mathcal{B}$:} For this, choose $R_0$ large enough so that $K\subset B_{R_0}$. Let $R>R_0$ large to be fixed below. Without loss of generality assume $0\in K$. Let $u_R$ be the $\mathcal{A}$-capacitary function for $B(0,R)\setminus K$. Using (b) in Lemma \ref{lemma3.3} we can choose $R$ sufficiently large so that
\[
|\nabla u_R(x)|\leq c_{\star} |x|^{\frac{1-n}{p-1}}=c_{\star} R^{\frac{1-n}{p-1}} \leq c.
\]
Therefore, $\Omega_1=B(0,R)\in \mathcal{B}$. \\

\noindent {\bf Existence of $\Omega_0\in\mathcal{G}_0$:} 
Let $R>0$ be as above and $u_R$ be the capacitary function for $B(0,R)\setminus K$. We first observe from the smoothness of $B(0,R)$ and \Lemma{lemma2.2HSI} that there is a constant $C>0$ and a neighbourhood  $U$ of $\partial K$ such that 
\[
|\nabla u_R|\geq C\, \, \mbox{in}\, \, U\setminus K.
\]
For a given $t$, $0<t<1$, let $\Omega_t=\{x\in B(0,R)\ : \, \, u_R(x)>1-t\}$. Then the capacitary function for $\Omega_t$ is 
\[
u_{\Omega_t}(x)=\frac{u_R-(1-t)}{t}.
\]
By choosing $t$ sufficiently small we have on $\partial\Omega_t$
\[
|\nabla u_{\Omega_t}|=\frac{|\nabla u_R|}{t}\geq \frac{C}{t}\geq c
\]
Therefore, $\Omega_0:=\Omega_t\in \mathcal{G}_0$ and $\overline{\Omega_0}\subset\Omega_1$

In view of these two observations and our earlier remarks, the existence of $\Omega$ is done. 

\subsubsection{Uniqueness of $\Omega$}
This will follow from \cite{HS}, where uniqueness was shown for the Laplace equation, and nonlinear elliptic differential equations satisfying properties (i)-(iv) given below, by using the Lavrent{\`e}v principle. In order to make use of this result for nonlinear elliptic equations, one needs to have four conditions (see section 4 in \cite{HS});
\begin{align}
\label{e:HScond}
\begin{array}{ll}
(i) & \mbox{The PDE is weakly elliptic and satisfies the comparison principle.}\\
(ii) & \mbox{If $u$ is a solution, then rotations and translations are also solutions to some}\\
& \quad \mbox{weakly elliptic PDE satisfying comparison principle.}\\
(iii) & \mbox{$u=x_n$ is a solution.}\\
(iv) & \mbox{If $\Omega$ and $K$ are both convex and  if $u_\Omega$ is the $\mathcal{A}-$capacitary function for}\\
& \mbox{$\Omega\setminus K$, then levels of $u$ are convex; 
$\Omega_t=\{x\in\Omega; \, \, u(x)=t\}$ is convex.}
\end{array}
\end{align}
Here (i) in \eqref{e:HScond} follows from the structural assumption on $\mathcal{A}$, (ii) follows from Lemma \ref{l:rotationissoln}. Regarding (iii), it is clear that $u=x_n$ is $\mathcal{A}$-harmonic, and (iv) follows from Lemma \ref{lemma3.4}. 

\subsubsection{Proof of $\Omega\in C^{1,\gamma}$.}
To obtain the $C^{1,\gamma}$ regularity of $\Omega$, one repeats the arguments of Vogel \cite{V}, which rely on the machinery of \cite{AC} and \cite{ACF}.

Furthermore, it follows from applying the Hodograph transform that if $\mathcal{A}\in C^{\infty}(\R^n\setminus \{0\})$, then $\partial\Omega\in C^{\infty}$, see \cite{KN77,KNS78}. We notice that an interesting alternative method to obtain higher regularity has recently been done in
\cite {DSS0}, where the authors prove higher order boundary Harnack estimates. See also \cite{DSS,DSS2} in the context of thin obstacle problems.

Now the proof of Theorem \ref{t:HS1} is complete.

\bibliographystyle{plain}

\bibliography{myref}  
\end{document}